\documentclass[11pt, one side]{article}
\usepackage[english]{babel}
\usepackage[centertags]{amsmath}
\usepackage{amssymb}
\usepackage{amsthm}
\usepackage{newlfont}
\usepackage{amsfonts}
\usepackage{bigints}
\usepackage{hyperref}
\usepackage{hypcap}
\theoremstyle{plain}

\newtheorem{remark}{Remark}

\newcommand{\R}{\mathbb{R}}
\newcommand{\CE}{\mathbb{C}}

\newcommand{\vl}{\left\lVert}
\newcommand{\vi}{\right\rVert}

\newcommand{\diag}{\text{diag}}

\newenvironment{psmallmatrix}{\left(\begin{smallmatrix}}{\end{smallmatrix}\right)}

\newtheorem{theorem}{Theorem}[section]

\newtheorem{corollary}{Corollary}[section]

\newtheorem{example}{Example}[section]
\newtheorem{lemma}{Lemma}[section]

\newcommand{\tref}[1]{Theorem~\textup{\ref{#1}}}

\newcommand{\cref}[1]{Corollary~\textup{\ref{#1}}}
\newcommand{\lref}[1]{Lemma~\textup{\ref{#1}}}
\newcommand{\rref}[1]{Remark~\textup{\ref{#1}}}

\begin{document}
\begin{center}
\textbf{\LARGE  Symmetric Norm Inequalities And Positive\vspace{4pt}\\ Semi-Definite Block-Matrices }
\end{center}
 \begin{center}
\emph{Antoine Mhanna$^1$ }
\end{center}
  \begin{center} \emph{1 Dept of Mathematics, Lebanese University, Hadath, Beirut, Lebanon.
}\end{center}
\begin{center}
\emph{tmhanat@yahoo.com}\end{center}

\centerline{\textbf{ Abstract}}\vspace{08pt}
 For positive semi-definite block-matrix  $M,$ we say that $M$ is P.S.D. and  we write  $M=\begin{pmatrix} A & X\\ {X^*} & B\end{pmatrix} \in {\mathbb{M}}_{n+m}^+$, with $A\in {\mathbb{M}}_n^+$, $B \in {\mathbb{M}}_m^+.$ The focus is on  studying  the consequences of a  decomposition lemma due to C.~Bourrin  and the main result is   extending the class of   P.S.D. matrices $M$ written by blocks of  same size  that satisfies the inequality: $\|M\|\le \|A+B\|$ for all symmetric norms.
\section{Introduction}
Let $A$ be an $n\times n$ matrix  and  $F$ an  $m\times m$ matrix, $(m>n)$ written by blocks such that  $A$ is a diagonal block and all entries other than those of $A$ are zeros, then the two matrices have the same  singular values and  for all unitarily invariant norms $\|A\|=\|F\|=\|A\oplus0\|$, we say then that the  symmetric norm on ${\mathbb{M}}_{m}$ induces a symmetric norm on ${\mathbb{M}}_{n}$, so for square matrices  we may assume that our norms are defined on all spaces ${\mathbb{M}}_n,$           $  n\ge 1.$
The spectral norm is denoted by ${\|.\|}_s,$ the Frobenius norm by ${\|.\|}_{(2)},$ and the Ky Fan $k-$norms by ${\|.\|}_{k}.$
Let ${\mathbb{M}}_n^+$ denote the set of positive and semi-definite part of the space of $n\times n$ complex matrices and $M$  be any positive semi-definite block-matrices; that is, 
$M=\begin{pmatrix} A & X\\ {X^*} & B\end{pmatrix} \in {\mathbb{M}}_{n+m}^+$, with $A\in {\mathbb{M}}_n^+$, $B \in {\mathbb{M}}_m^+.$
\section{Decomposition of block-matrices}
\begin{lemma}\label{prince}
For every matrix $M$  in  ${\mathbb{M}}_{n+m}^+$ written in blocks, we have the decomposition:
$\begin{pmatrix} A & X\\ {X^*} & B\end{pmatrix}=U\begin{pmatrix} A & 0\\ {0} & 0 \end{pmatrix}U^*+V\begin{pmatrix} 0 & 0\\ {0} & B \end{pmatrix}V^*$\\
for some unitaries $U, V\in {\mathbb{M}}_{n+m}.$
\end{lemma}
\begin{proof}
Factorize the positive matrix as a square of positive matrices:$$
\begin{pmatrix} A & X\\ {X^*} & B\end{pmatrix}=\begin{pmatrix} C & Y\\ {Y^*} & D\end{pmatrix}.\begin{pmatrix} C & Y\\ {Y^*} & D\end{pmatrix}$$
we verify that the right hand side can be written as $ T^*T+S^*S$ so :$$\begin{pmatrix} C & Y\\ {Y^*} & D\end{pmatrix}.\begin{pmatrix} C & Y\\ {Y^*} & D\end{pmatrix}= \underbrace{\begin{pmatrix} C & 0\\ {Y^*} & 0\end{pmatrix}}_{T^*}.\underbrace{\begin{pmatrix} C & Y\\ {0} & 0\end{pmatrix}}_{T}+\underbrace{\begin{pmatrix} 0 & Y\\ {0} & D\end{pmatrix}}_{S^*}.\underbrace{\begin{pmatrix} 0 & 0\\ {Y^*} & D\end{pmatrix}}_{S}.$$ Since  $TT^*=\begin{pmatrix}CC+YY^*&0\\0&0  \end{pmatrix}=\begin{pmatrix} A & 0\\ {0} & 0 \end{pmatrix},$  $SS^*=\begin{pmatrix} 0 & 0\\ 0& Y^*Y+DD \end{pmatrix}=\begin{pmatrix} 0 & 0\\ {0} & B \end{pmatrix}$ and 
$AA^*$ is unitarily congruent to $A^*A$ for any square matrix $A,$	 the lemma follows.
 \end{proof}
\begin{remark}\label {impol}As a consequence of this lemma we have:$$\|M\|\le \|A\| +\|B\|$$  for all symmetric norms.
\end{remark}
Equations involving unitary matrices  are called unitary orbits representations. Recall that if $A\in {\mathbb{M}}_n$,  $R(A)=\dfrac{A+A^*}{2}$ and $I(A)=\dfrac{A-A^*}{2i}.$
\begin{corollary}\label{mtlo}
For every matrix in ${\mathbb{M}}_{2n}^+ $ written in blocks of the same size, we have the  decomposition: 
$$\begin{pmatrix} A & X\\ {X^*} & B\end{pmatrix}=U\begin{pmatrix} \frac{A+B}{2}-R(X) & 0\\ {0} & 0 \end{pmatrix}U^*+V\begin{pmatrix} 0 & 0\\ {0} &\frac{A+B}{2}+R(X) \end{pmatrix}V^*$$
for some unitaries $U, V\in {\mathbb{M}}_{2n}.$
\end{corollary}
\begin{proof}
Let $J=\dfrac{1}{\sqrt{2}}\begin{pmatrix} I & -I\\ {I} & I\end{pmatrix}$ where $I$ is the identity of ${\mathbb{M}}_n$, $J$ is a unitary matrix, and we have:$$
J\begin{pmatrix} A & X\\ {X^*} & B\end{pmatrix}J^*=\underbrace{\begin{pmatrix} \frac{A+B}{2}-R(X) & \frac{A-B}{2}+\frac{ X^*-X}{2}\\[0.5cm]  \frac{A-B}{2}-\frac{X-X^*}{2} & \frac{A+B}{2}+R( X)\end{pmatrix}}_N$$
Now we factorize $N$ as a square of positive matrices:
$$\begin{pmatrix} A & X\\ {X^*} & B\end{pmatrix}=J^*\begin{pmatrix} L & M\\ {M^*} & F\end{pmatrix}.\begin{pmatrix} L & M\\ {M^*} & F\end{pmatrix}J$$
and let:
$$ \begin{array}{rcl}\delta &=J^*\begin{pmatrix} L & M\\ {M^*} & F\end{pmatrix}& =\frac{1}{\sqrt{2}}\begin{pmatrix} L+M^* & M+F\\ {M^*}-L & F-M\end{pmatrix}\\ \psi &=\begin{pmatrix} L & M\\ {M^*} & F\end{pmatrix}J&= \frac{1}{\sqrt{2}}\begin{pmatrix} L+M & M-L\\ F+{M^*} & F-{M^*}\end{pmatrix}\end{array}$$
A direct computation shows that: \begin{align}\delta .\psi &=\dfrac{1}{2}\begin{psmallmatrix}(L+M^*)(L+M)+(M+F)(F+M^*)\text{    } &\text{     } (L+M^*)(M-L)+(M+F)(F-M^*)\\[0.7cm] (M^*-L)(L+M)+(F-M)(F+M^*)\text{         } &\text{    } (M^*-L)(M-L)+(F-M)(F-M^*)\end{psmallmatrix}\nonumber\\&={\Gamma}^*\Gamma + {\Phi}^*\Phi \end{align}where: $\Gamma=\dfrac{1}{\sqrt{2}}\begin{pmatrix}L+M & M-L\\ 0 &0\end{pmatrix},$ and $ {\Phi}=\dfrac{1}{\sqrt{2}}\begin{pmatrix}0& 0\\F+M^* & F-M^*\end{pmatrix} $  to finish notice that for any square matrix $A$,  ${A^*A} $ is unitarily congruent to ${A}{A }^*$
and, ${\Gamma}{\Gamma}^*$, $ {\Phi}{\Phi}^*$  have the required form.
\end{proof}
The previous corollary implies that $\frac{A+B}{2}\ge R(X)$ and  $\frac{A+B}{2}\ge -R(X).$    
\begin{corollary}\label{woa}
For every matrix in ${\mathbb{M}}_{2n}^+ $ written in blocks of the same size, we have the  decomposition: $$\begin{pmatrix} A & X\\ {X^*} & B\end{pmatrix}=U\begin{pmatrix} \frac{A+B}{2}+I(X) & 0\\ {0} & 0 \end{pmatrix}U^*+V\begin{pmatrix} 0 & 0\\ {0} &\frac{A+B}{2}-I(X) \end{pmatrix}V^*$$
for some unitaries $U, V\in {\mathbb{M}}_{n+m}.$
\end{corollary}
\begin{proof}
The proof is similar to \cref{mtlo}, we have:  $J_1\begin{pmatrix} A &X \\ X^* & B \end{pmatrix}J_1^*=\begin{pmatrix}A & iX \\-iX^*& B \end{pmatrix}$ where $J_1=\begin{pmatrix}I & 0       \\0& -i I
\end{pmatrix}$, and
 $$ K=JJ_1\begin{pmatrix} A &X \\ X^* & B \end{pmatrix}J_1^*J^*= \begin{pmatrix} \frac{A+B}{2}+I(X) & *\\[0.4cm] * & \frac{A+B}{2}-I(X)\end{pmatrix}  $$      here (*) means an unspecified entry, the proof is similar to that in \cref{mtlo} but for reader's  convenience we give the main headlines:  first factorize $ K$ as a square of positive matrices; that is,  $M=\begin{pmatrix} A &X \\ X^* & B \end{pmatrix}=J_1^*J^* L^2JJ_1$ next  decompose $L^2$ as in \lref{prince} to obtain $$M=J_1^*J^* ( T^*T+S^*S )JJ_1=J_1^*J^* ( T^*T)JJ_1+J_1^*J^*(S^*S )JJ_1$$  where  $TT^*=\begin{pmatrix} \frac{A+B}{2}+I(X) & 0\\ {0} & 0
\end{pmatrix}$   and $SS^*= \begin{pmatrix}0 & 0\\ {0} &\frac{A+B}{2}-I(X)
\end{pmatrix}$
finally the congruence property completes the proof.
\end{proof}
The existence of unitaries $U$ and $V$ in the decomposition process need not to be unique as one can take the special case; that is, $M$   any diagonal matrix with  diagonal entries equals a nonnegative number $k$,   explicitly
$M=kI=U\left(\dfrac{k}{2}I\right)U^*+ V\left(\dfrac{k}{2}I\right)V^*$ for any $U$ and $V$ unitaries. 
\begin{remark}
Notice that from the Courant-Fischer theorem if   $A, B \in {\mathbb{M}}_n^+$, then the eigenvalues of each matrix are the same as the singular values and $A\le B \Longrightarrow \|A\|_k\le \|B\|_k$, for all $k=1,\cdots,n$,
also $A<B \Longrightarrow \|A\|_k< \|B\|_k$, for all $k=1,\cdots,n.$
\end{remark}
\begin{corollary}
For every matrix in ${\mathbb{M}}_{2n}^+ $ written in blocks of the same size, we have: 
$$\begin{pmatrix} A & X\\ {X^*} & B\end{pmatrix}\le\dfrac{1}{2}\bigg\lbrace
 U \begin{pmatrix} A+B+|X-X^*| & 0\\ {0} & 0 \end{pmatrix}U^*+V\begin{pmatrix} 0 & 0\\ {0} &A+B+|X-X^*|\end{pmatrix}V^*\bigg\rbrace$$
for some unitaries $U, V\in {\mathbb{M}}_{n+m}.$
\end{corollary}
\begin{proof}
This a consequence of the fact that $I(X)\le |I(X)|$.
\end{proof}
\section{Symmetric Norms  and Inequalities}
In    \cite{1} they found that if $X$ is hermitian then  \begin{equation}\label{pkm}\|M\|\le \|A+B\|\end{equation} for all symmetric norms. It has been given counter-examples  showing that this does not necessarily holds if $X$ is a  normal but not  Hermitian matrix, the main idea of this section is  to give examples and counter-examples in a general way and to extend the previous inequality to a larger class of P.S.D. matrices written by blocks satisfying \eqref{pkm}.
\begin{theorem}\label{poiu}
If $A$ and $B$ are  positive definite matrices of same size. Then $$
\begin{pmatrix}A&X \\X^*&B\end{pmatrix}>0 \Longleftrightarrow A \ge X B^{-1}X^*$$
\end{theorem}
\begin{proof}
Write $\displaystyle \begin{pmatrix} A &X \\X^* &B \end{pmatrix}=\begin{pmatrix} I &XB^{-1} \\ 0&I \end{pmatrix}\begin{pmatrix} A-XB^{-1}X^* &0\\0 &B \end{pmatrix}\begin{pmatrix} I &0 \\XB^{-1} &I \end{pmatrix}$
where $I$ is the identity matrix, and that complete the proof since for any matrix $A,$  $$A\ge 0 \Longleftrightarrow X^*AX\ge 0,\text{      } \forall X.$$ 
\end{proof}
\begin{theorem}\label{sinj}
Let $M=\begin{pmatrix}A&B\\C&D\end{pmatrix}$ be any square matrix written by blocks of same size, if $AC=CA$ then $\det(M)=\det(AD-CB)$
\end{theorem}
\begin{proof}
Suppose first that $A$ is invertible, let us write $M$ as \begin{equation}\label{chou}M=\begin{pmatrix}Z&0\\V&I\end{pmatrix}\begin{pmatrix}I&E\\0&F\end{pmatrix}\end{equation} upon calculation we find that: $Z=A,$ $  V=C,$ $ E=A^{-1}B,$ $F=D-CA^{-1}B$
taking the determinant on each side of \eqref{chou}  we get: $$\det(M)=\det(A(D-CA^{-1}B))=\det(AD-CB)$$ the result follows by a continuity argument since the Determinant function is a  continuous function.
\end{proof}
Given the matrix  $M=\begin{pmatrix} A & X\\ {X^*} &0\end{pmatrix}$ a  matrix in  ${\mathbb{M}}_{2n}^+ $  written by blocks of same size,   we know that it $M$ is not P.S.D., to see this   notice  that all  the $2\times 2 $  extracted principle submatrices of $M$  are  P.S.D if  and only if $X=0 $ and $A$	 is positive semi-definite.
Even if a proof of this exists and  would take two lines, it is quite nice to see a different constructive proof, a direct consequence of \lref{prince}.
\begin{theorem}\label{mhan1}
 Given  $\begin{pmatrix} A & X\\ {X^*} & B\end{pmatrix}$ a  matrix in  ${\mathbb{M}}_{2n}^+ $  written in blocks of same size:\begin{enumerate}
\item If $\begin{pmatrix} A & X\\ {X^*} &0\end{pmatrix}$ is positive semi-definite,   $I(X)>0$ or $I(X)<0$, then there exist a matrix $Y$ such that $M=\begin{pmatrix} A & Y\\ {Y^*} &0\end{pmatrix}$ is positive semi-definite and: 
\begin{equation}\label{pls}\left\lVert\begin{pmatrix} A & Y\\ {Y^*} & 0\end{pmatrix}\right\rVert>\left\lVert A\right\rVert  \end{equation}
 for all symmetric norms.  
\item If $\begin{pmatrix} 0 & X\\ {X^*} &B\end{pmatrix}$ is positive semi-definite,   $I(X)>0$ or $I(X)<0$  then there exist a matrix $Y$ such that $M=\begin{pmatrix} 0 & Y\\ {Y^*} &B\end{pmatrix}$ is positive semi-definite and: 
\begin{equation}\label{pla}\left\lVert\begin{pmatrix} 0 & Y\\ {Y^*} & B\end{pmatrix}\right\rVert>\left\lVert B\right\rVert  \end{equation}
\end{enumerate}
The same result holds if we replaced $I(X)$ by $R(X)$ because $\begin{pmatrix} A & iX\\ -i{X^*} & B\end{pmatrix}$ is unitarily congruent to  $\begin{pmatrix} A & X\\ {X^*} & B\end{pmatrix}.$ 
\end{theorem}
\begin{proof}
Without loss of generality we can consider $I(X)>0$ cause $\begin{pmatrix} A & X\\ {X^*} & B\end{pmatrix}$ and $\begin{pmatrix} A & -X\\ -{X^*} & B\end{pmatrix}$  are unitarily congruent, we will show the first statement as the second one has a similar proof, from \cref{woa} we have:  $$\begin{pmatrix} A & X\\ {X^*} &0\end{pmatrix}\ge U\begin{pmatrix} \frac{A}{2} & 0\\ {0} & 0 \end{pmatrix}U^*+U\begin{pmatrix}I(X) & 0\\ {0} & 0  \end{pmatrix}U^*+ V\begin{pmatrix} 0 & 0\\ {0} &\frac{A}{2}\end{pmatrix}V^*$$ 
Since $\begin{pmatrix} A & X\\ {X^*} &0\end{pmatrix}$ is congruent to $L=\begin{pmatrix} A & lX\\ {lX^*} &0\end{pmatrix}$   for any $l\in \CE$, $L$ is P.S.D.
$A$ is a fixed matrix, we have ${\vl U\begin{pmatrix} \frac{A}{2} & 0\\ {0} & 0 \end{pmatrix}U^*+V\begin{pmatrix} 0 & 0\\ {0} &\frac{A}{2}\end{pmatrix}V^*\vi}_k=\beta \|A\|_k$ for some $\beta\le 1$
finally we set $Y=lX$  where $l\in \R$  is large enough to have $ {\|M\|}_k > \|A\|_k,$  $ \forall k$ thus ${\|M\|} > \|A\|$ for all symmetric norms.
\end{proof}
Notice  that there exist a permutation matrix $P$ such that $ P\begin{pmatrix}A&X\\X^*&0  \end{pmatrix} P^{-1}=\begin{pmatrix}0&X^*\\X&A   \end{pmatrix}$ and since $I(X)>0$ if and only if $I(X^*)<0,$ the two assertions of  \tref{mhan1} are equivalent up to a permutation similarity.
\begin{corollary}\label{cvb}
If $M=  \begin{pmatrix} A & X\\ {X^*} & 0\end{pmatrix},$ $A$  a positive semi-definite matrix, and we have one of the following conditions:
\begin{enumerate}
\item $R(X) > 0$
\item $R(X) <0$
\item $I(X) >0$
\item $I(X) < 0 $
\end{enumerate}
Then   $M$  can't be   positive semi-definite. 
\end{corollary}
\begin{proof}
By \rref{impol} any positive semi-definite matrix $M$ written in blocks must satisfy $\|M\| \le\|A\|+\|B\|$ for all symmetric norms which is not the case of the matrix  $M$  constructed in  \tref{mhan1}. 
\end{proof}
Finally we get:
\begin{theorem}\label{mhan2}
If $X\neq 0$ and $B=0$, $A\ge 0$, the matrix $M=  \begin{pmatrix} A & X\\ {X^*} & 0\end{pmatrix}$                  cannot be positive semi-definite. 
\end{theorem}
 \begin{proof}
Suppose the converse, so $M=  \begin{pmatrix} A & X\\ {X^*} & 0\end{pmatrix}$   is positive semi-definite, without loss of generality the only case   we need to discuss is when $	R(X)$ has positive and negative eigenvalues, by \cref{mtlo} we can write: $$M=U\begin{pmatrix} \frac{A}{2}-R(X) & 0\\ {0} & 0 \end{pmatrix}U^*+V\begin{pmatrix} 0 & 0\\ {0} &\frac{A}{2}+R(X) \end{pmatrix}V^*$$
for some unitaries $U, V\in {\mathbb{M}}_{2n}.$
Now if $R(X)$ has $-\alpha$ the smallest negative eigenvalue $R(X)+(\alpha+\epsilon) I> 0 $ consequently  the matrix \begin{align}
H&=U\begin{pmatrix} \frac{A}{2}-R(X) & 0\\ {0} & 0 \end{pmatrix}U^*+V\begin{pmatrix} 0 & 0\\ {0} &\frac{A}{2}+(\alpha+\epsilon) I+R(X)+ (\alpha+\epsilon) I\end{pmatrix}V^*\\&= \begin{pmatrix} A+2(\alpha+\epsilon) I & X+(\alpha+\epsilon) I\\ (X+(\alpha+\epsilon) I)^* & 0\end{pmatrix}\end{align}
is positive semi-definite with $R(Y)>0$,  where $Y=X+(\alpha+\epsilon) I$, by \cref{cvb} this is a contradiction.
\end{proof}
A natural question would be  how many are the  nontrivial  P.S.D.matrices written by blocks ? The following lemma will show us how to construct some of them.
\begin{lemma}
Let $A$ and $B$ be any $n\times n$ positive definite matrices, then there exist an integer ${t}\ge 1 $ such that the matrix $F_{t}=\begin{pmatrix}tA&X\\X^*&\text{\small{t}}\normalsize{B}\end{pmatrix}$
is positive definite. 
\end{lemma}
\begin{proof}
Recall from \tref{poiu} that  $F_{1}$ is positive definite if and only if $A > X B^{-1}X^*,$ which is equivalent to  $ x^*Ax > x^*XB^{-1}X^*x $  for all $ x\in {\CE}^n.$  Set $f(x):= x^*Ax $ and $g(x):=x^*XB^{-1}X^*x$ and let us suppose, to the contrary, that there exist a vector $z$ such that $f(z)\le g(z)$  since $f(x)$ and $g(x)$ are homogeneous functions  of degre $d=2$ over $\R$  if
 $ f(x) \ge g(x)$ for all $x$ such that ${\|x\|}_s=1$ then  $f(x) \ge g(x)$ for any $x\in{\CE}^n.$ So let us set $K=\underset{{\|x\|}_s=1}{\max }g(x),$ and $L=\underset{{\|x\|}_s=1}{\min}f(x)$ 
since $ g(x)$  and $f(x)$ are continuous functions and $ \{x;{\|x\|}_s=1\}$ is  compact, there exist a vector $w$ respectively $v$ such that $K=g(w),$ respectively  $L=f(v).$  Now choose $t\ge1$ such that $tf(v)>\dfrac{ g(w)}{t},$ to obtain $$ x^*(tA)x\ge v^*(tA)v >w^*X(tB)^{-1}X^*w\ge x^*X(tB)^{-1}X^*x$$ for all $x$ such that ${\|x\|}_s=1$, thus $ x^*(tA)x> x^*X(tB)^{-1}X^*x$ for any $x\in 
{\CE}^n$  which completes the proof.
\end{proof}\begin{theorem}\label{pww}
Let $A=\diag({\lambda}_1,\cdots, {\lambda}_n),$  $B=\diag({\nu}_1,\cdots, {\nu}_n)$ and   $M=\begin{pmatrix}A&X\\X^*&B    \end{pmatrix}$ a given  positive semi-definite matrix. If  $X^*$ commute with $A$ and $X^*X$ 	equals a diagonal matrix, then  $$\|M\| \le \|A+B\|$$  for all symmetric norms. The same inequality holds if  $X$ commute with $B$ and $XX^*$ is diagonal.
\end{theorem}
\begin{proof}
It suffices to prove the inequality for the  Ky Fan $k-$norms  $k=1,\cdots,n,$ let $P=\begin{pmatrix}0&I_{n}\\I_{n}&0\end{pmatrix}$  where $I_{n}$ is the identity matrix of order $n$, since  $\begin{pmatrix}B&X^*\\X&A \end{pmatrix}=P\begin{pmatrix}A&X\\X^*&B    \end{pmatrix}P^{-1}$  and $\begin{pmatrix}A&X\\X^*&B \end{pmatrix}$ have same singular values,  we will discuss only the first case; that is,  when $X^*$ commute with $A$ and $X^*X$ is diagonal, as the second case will follows. Let $D:=X^*X=\begin{psmallmatrix}d_{1} & 0 &\cdots & 0\\ 0 &d_{2} & \cdots & 0\\ \vdots & \vdots & \ddots &\vdots \\
0& 0 & \cdots &d_{n} \end{psmallmatrix},$ as $X^*$ commute with $A,$ from \tref{sinj} we conclude that the eigenvalues of   $\begin{pmatrix}A&X\\X^*&B \end{pmatrix}$ are the roots of $$
 \det ((A-{\mu}I_{n})(B-{\mu}I_{n})-D)=0 $$
Equivalently the eigenvalues are all the solutions of  the $n$ equations:$$\begin{array}{rlrcl}1)&&({\lambda}_1-{\mu})({\nu}_1-{\mu})-d_{1}&=&0
\\[0.1cm]2)&&({\lambda}_2-{\mu})({\nu}_2-{\mu})-d_{2}&=&0
\\[0.1cm]3)&&({\lambda}_3-{\mu})({\nu}_3-{\mu})-d_{3}&=&0\\[0.1cm]\vdots{\text{ }}&&&\vdots&\\[0.1cm]
i)&&({\lambda}_i-{\mu})({\nu}_i-{\mu})-d_{i}&=&0\\\vdots\text{  }&&&\vdots&\\[0.1cm]n)&&({\lambda}_n-{\mu})({\nu}_n-{\mu})-d_{n}&=&0
\end{array}$$
Each equation is of $2^{nd}$ degree, if we denote by $a_i$ and $b_i$ the two solutions of the $i^{th}$ equation we  deduce that:
$$\begin{array}{rcl}a_1+b_1&=&{\lambda}_1+{\nu}_1\\[0.1cm]
a_2+b_2&=&{\lambda}_2+{\nu}_2\\ &\vdots&\\[0.1cm]
a_n+b_n&=&{\lambda}_n+{\nu}_n\\[0.1cm]
\end{array}$$
But $$A+B=\begin{pmatrix}{\lambda}_1+{\nu}_1 & 0 &\cdots & 0\\ 0 &{\lambda}_2+{\nu}_2 & \cdots & 0\\ \vdots & \vdots & \ddots &\vdots \\
0& 0 & \cdots &{\lambda}_n+{\nu}_n  \end{pmatrix}$$
and each diagonal entry of $A+B$ is equal the sum of two nonegative  eigenvalues of $M$, thus we have necessarily: $\|M\|_k\le \|A+B\|_k$ for all $k=1,\cdots,n$ which completes the proof.
\end{proof}
\begin{example}
Let $$M_x=\begin{pmatrix}  x& 0&\dfrac{i}{2}&0\\0&\dfrac{99}{100}&0&-\dfrac{i}{2}\\
-\dfrac{i}{2}& 0& \dfrac{99}{100}&0\\
0&\dfrac{i}{2}& 0&\dfrac{1}{2}\end{pmatrix}$$
 If  $\dfrac{3}{10}\le x\le \dfrac{1}{2}$,  $M_x$ is positive definite and we have: \begin{equation}\label{lkl}\|M_x\|_{}\le \|A+B\|\end{equation}  for all symmetric norms, where $A=\begin{pmatrix}  x& 0\\0&\frac{99}{100} \end{pmatrix}$ and $B=\begin{pmatrix}\frac{99}{100} &0\\0&\frac{1}{2} \end{pmatrix}.$  
If $M_x$ is positive definite for $x=\dfrac{3}{10}$  then  $M_x$ is P.D. for all  $x>\dfrac{3}{10}$.
The eigenvalues of $M_{\frac{3}{10}}$ which are the same as the singular values of $M_{\frac{3}{10}}$ are:\begin{align}{\lambda}_1&=\frac{149}{200}+\frac{\sqrt{12401}}{200}\approx 1.301 \\[0.2cm]{\lambda}_2&=  \frac{129}{200}+\frac{\sqrt{14761}}{200}\approx 1.25  \\[0.2cm]{\lambda}_3&=\frac{149}{200}-\frac{\sqrt{12401}}{200}\approx 0.188 \\[0.2cm]{\lambda}_4&= \frac{129}{200}-\frac{\sqrt{14761}}{200}\approx 0.0375
\end{align} And the \eqref{lkl} inequality  follows from \tref{pww}. 
\end{example}
Let us study the commutation condition in \tref{pww}. First  notice that any square matrix $X =(x_{ij})\in {\mathbb{M}}_n $ will commute with $A=\diag(a_{1},\cdots, a_{n})$  if and only if :
$$Y'=\begin{pmatrix}x_{1,1}a_{1} & x_{1,2}a_2 &\cdots & x_{1,n}a_n\\x_{2,1}a_{1} & x_{2,2}a_2& \cdots & x_{2,n}a_n\\ \vdots & \vdots & \ddots &\vdots \\
x_{n,1}a_1& x_{n,2}a_2 & \cdots &x_{n,n} a_{n} \end{pmatrix}=\begin{pmatrix}x_{1,1}a_{1} & x_{1,2}a_1 &\cdots & x_{1,n}a_1\\x_{2,1}a_{2} &x_{2,2}a_2 &\cdots &x_{2,n}a_2\\ \vdots & \vdots & \ddots &\vdots \\
x_{n,1}a_n& x_{n,2}a_n & \cdots &x_{n,n} a_{n} \end{pmatrix}=Y$$
An $(i,j)$ entry of $Y'$ is equal to that of $Y$ if and only if $x_{i,j}a_j=x_{i,j}a_i,$ i.e. either $a_i=a_j$ or $x_{i,j}=0.$
\begin{corollary}
Let $A=\diag({\lambda}_1,\cdots, {\lambda}_n),$  $B=\diag({\nu}_1,\cdots, {\nu}_n)$  and   $M=\begin{pmatrix}A&X\\X^*&B    \end{pmatrix}$  a given positive semi-definite matrix. If  $X^*$ commute with $A$, or   $X$ commute with $B$, then  $\text{           }  \|M\| \le \|A+B\| \text{         }$ for all symmetric norms.
\end{corollary}
\begin{proof}
As in \tref{pww}, we will assume  without loss of generality that $X^*$ commute with $A,$ as the other case is similar. If $X^*$ is diagonal the result follows from \tref{pww}, suppose there is an off diagonal entry $x_{i,j}$ of $X^*$ different from $0$, from the commutation condition we have $a_i=a_j$ and   the same goes  for all such entries, of course if $AX=XA$ then $$PAXP^{-1}=PXAP^{-1}=PAP^{-1}PXP^{-1}=PXP^{-1}PAP^{-1}=PXAP^{-1}$$  Take $P$  to be the permutation matrix that will order the same diagonal entries of $A$ in a  one diagonal  block and keeps the  matrix $B$ the same,   since $M$  is Hermitian so is $PMP^{-1}$  because we  can consider the   permutation matrix as a product of transposition  matrices  $P_1,\cdots, P_n$ wich are orthogonal;  in other words 
$$PMP^{-1}=P_1P_2\cdots P_n MP_n^{T}\cdots P_2^{T}P_1^{T}.$$ Consequently $P^{T}=P^{-1}$  for any permutation matrix and   $\|M\|=\|PMP^{T}\|$ for all symmetric norms. If $H= PMP^{T},$ $D:=PX$  and $X_i$ is  some $i\times i$ extracted  submatrix of  $X^*$, we will have the block written matrix  $$H=\begin{pmatrix}PAP^{T}&PX\\[0.2cm]{X^*P^{T}} &B\end{pmatrix}=\begin{pmatrix}\begin{psmallmatrix}aI_i & O_j &\cdots & O_s\\ O_i &bI_{j} & \cdots & O_s\\ \vdots & \vdots & \ddots &\vdots \\
O_i&O_j & \cdots &rI_s\end{psmallmatrix}& \begin{psmallmatrix}     X_i^* & O_i &\cdots & O_i\\ O_j &X_{j}^* & \cdots & O_j\\ \vdots & \vdots & \ddots &\vdots \\
O_s&O_s & \cdots &X_{s}^* \end{psmallmatrix}\\[0.8cm] \begin{psmallmatrix}X_i & O_j &\cdots & O_s\\ O_i &X_{j} & \cdots & O_s\\ \vdots & \vdots & \ddots &\vdots \\
O_i&O_j & \cdots &X_s \end{psmallmatrix} &\begin{psmallmatrix}{\nu}_{1} & 0 &\cdots & 0\\ 0 &{\nu}_{2} & \cdots & 0\\ \vdots & \vdots & \ddots &\vdots \\
0& 0 & \cdots &{\nu}_{n}\end{psmallmatrix}\end{pmatrix}$$ where we denoted   the   diagonal matrix of order $i$  whose  diagonal entries are  equal to $a$  by  $aI_{i}$ and  the  zero block of order $i$ by  $O_i.$ Let us  calculate the roots  of the characteristic polynomial of $H$; that is, the roots of
$$\det\left(\begin{psmallmatrix}(a-{\lambda})I_i& O_j &\cdots & O_s\\ O_i &(b-{\lambda})I_{j} & \cdots & O_s\\ \vdots & \vdots & \ddots &\vdots \\
O_i&O_j & \cdots &(r-{\lambda})I_s\end{psmallmatrix}     \begin{psmallmatrix}{\nu}_{1}-{\lambda} & 0 &\cdots & 0\\ 0 &{\nu}_{2}-{\lambda}& \cdots & 0\\ \vdots & \vdots & \ddots &\vdots \\
0& 0 & \cdots &{\nu}_{n}-{\lambda}\end{psmallmatrix}  -    D^*D \right)=0$$  we  translate this to a system of blocks, while each  eigenvalue of $H,$ which is the same as its singular value, will verify one of the following  equations:
$$ \begin{array}{rlrcl}1) & \det\left((a-{\lambda})I_i)\big(\begin{psmallmatrix}{\nu}_{1}-{\lambda} &\cdots & 0\\ \vdots &  \ddots &\vdots \\
0&  \cdots &{\nu}_{i}- {\lambda}              \end{psmallmatrix}\big)-X_{i}^*{X}_i\right)& &=&0
\\[0.4cm]2)&\det\left((b-{\lambda})I_j)\big( \begin{psmallmatrix}{\nu}_{i+1}-{\lambda} &\cdots & 0\\ \vdots &  \ddots &\vdots \\
0&  \cdots &{\nu}_{i+j}- {\lambda}             \end{psmallmatrix}\big)-X_{j}^*{X}_j\right)& &=&0
\\\vdots{\text{ }}&&&\vdots&\\[0.4cm]
c)&\det\left((r-{\lambda})I_s)\big(\begin{psmallmatrix}{\nu}_{n-s}-{\lambda} &\cdots & 0\\ \vdots &  \ddots &\vdots \\
0&  \cdots &{\nu}_{n}- {\lambda} \end{psmallmatrix}\big)-X_{s}^*{X}_s\right)&&=&0
\end{array}\indent (T)\\[0.1cm]$$
where $c$ is the number of diagonal blocks we have. Let us have a closer look to any of the equations above,  without loss of generality we will take the first one, the same will hold for the others, notice   that all eigenvalues ${\lambda}$ are nonnegative and we have  $$M_1=\begin{pmatrix}aI_i& X_i^*\\X_i&\begin{psmallmatrix}{\nu}_{1}&\cdots & 0\\ \vdots &  \ddots &\vdots \\ 0&  \cdots &{\nu}_{i}   \end{psmallmatrix}
\end{pmatrix}=\begin{pmatrix} C_1& X_i^*   \\  X_i^*& K_1         \end{pmatrix} $$ is positive semi-definite because it's eigenvalues are a subset of those  of $M.$ The key idea is that for this matrix $\|C_1+K_1\|=\|C_1\|+\|K_1\|$ for all symmetric norms.
where $C_1= aI_i   $
and $K=\begin{psmallmatrix}{\nu}_{1}&\cdots & 0\\ \vdots &  \ddots &\vdots \\
0&  \cdots &{\nu}_{i}   \end{psmallmatrix}.$ Now back to the system $(T)$ we associate like we did to $M_1$  each equation whose number is  $i$ to  a positive semi-definite matrix $M_i$ 
to obtain by \rref{impol}  $$\begin{array}{rclcl} \|M_1\|_k &\le& \left\|aI_i+\begin{psmallmatrix}{\nu}_{1}&\cdots & 0\\ \vdots &  \ddots &\vdots \\ 0&  \cdots &{\nu}_{i}   \end{psmallmatrix}     \right\|_k&=& \|aI_i\|_k+\left\| \begin{psmallmatrix}{\nu}_{1}&\cdots & 0\\ \vdots &  \ddots &\vdots \\ 0&  \cdots &{\nu}_{i}   \end{psmallmatrix}     \right\|_k
\\[0.65cm] \|M_2\|_k &\le& \left\|bI_j+\begin{psmallmatrix} {\nu}_{i+1}&\cdots & 0\\ \vdots &  \ddots &\vdots \\ 0&  \cdots &{\nu}_{i+j}    \end{psmallmatrix}   \right\|_k &=&\|bI_j\|_k+\left\| \begin{psmallmatrix} {\nu}_{i+1}&\cdots & 0\\ \vdots &  \ddots &\vdots \\ 0&  \cdots &{\nu}_{i+j}    \end{psmallmatrix}   \right\|_k
\\[0.7cm]&\vdots&&\vdots&
\\[0.3cm]\|M_c\|_k &\le& \left\|rI_s+\begin{psmallmatrix}{\nu}_{n-s}&\cdots & 0\\ \vdots &  \ddots &\vdots \\ 0&  \cdots &{\nu}_{n}   \end{psmallmatrix}      \right \|_k&=&\|rI_s\|_k+\left\|\begin{psmallmatrix}{\nu}_{n-s}&\cdots & 0\\ \vdots &  \ddots &\vdots \\ 0&  \cdots &{\nu}_{n}   \end{psmallmatrix}  \right \|_k
\end{array}$$
for all $k,$ but the order of the entries of $B$ are arbitrary chosen,  thus from \tref{pww}  $\|M\|_k\le \|A+B\|_k$ for all $k=1,\cdots,n$ and that completes the proof.
\end{proof}
\begin{corollary}
Let $M=\begin{pmatrix}A&X\\X^*&B    \end{pmatrix}$ be a positive semi-definite matrix written by blocks. There exist a unitary $V$ and a unitary $U$ such that$$ \left\lVert\begin{pmatrix} A & X\\ {X^*} & B\end{pmatrix}\right\rVert \le\|UAU^*+VBV^*\|:=\|A\|+\|B\|$$ for all symmetric norms.
\end{corollary}
\begin{proof}
Let $U$  and $V$ be two  unitary matrix such that $UAU^*=D_{o}$ and $VBV^*=G_{o}$ where $D_{o}$  and $G_{o}$ are two diagonal matrices having the same ordering ${o},$ of eigenvalues with respect to their indexes  i.e.,  if ${\lambda}_n\le\cdots\le{\lambda}_1$ are the diagonal entries of $D_{o}$, and  ${\nu}_n\le\cdots\le{\nu}_1$ are those of  $G_{o}$, then if  ${\lambda}_i$ is in the $(j,j)$ position then ${\nu}_i$ will be also. Consequently  $\|UAU^*+VBV^*\|=\|D_{o}+G_{o}\|=\|D_{o}\|+\|G_{o}\|=\|A\|+\|B\|,$ for all the Ky-Fan $k-$norms and thus for all symmetric norms. To complete the proof notice  that if $T=UXV^*$  and $Q$ is the unitary matrix $\begin{pmatrix}U&0\\0&V \end{pmatrix},$ by \rref{impol} \begin{align}\left\lVert\begin{pmatrix}A&X\\X^*&B \end{pmatrix}\right\rVert=\left\lVert Q\begin{pmatrix}A&X\\X^*&B \end{pmatrix}Q^*\right\rVert=\left\lVert\begin{pmatrix}D_{o}&T\\{T}^*&G_{o}\end{pmatrix}\right\rVert&\le \|D_{o}\|+\|G_{o}\|\end{align}
 for all symmetric norms.
\end{proof}
\begin{theorem}
Let $M=\begin{pmatrix}A&X\\X^*&B    \end{pmatrix}\ge0,$ if $X$ is normal, $X^*$ commute with $A$ and $X$ commute with $B,$ then we have $\|M\|\le \|A+B\|$ for all symmetric norms.
\end{theorem}
\begin{proof}
We consider first that the normal matrix $X^*$ has all of its eigenvalues distinct,  by \tref{l} and the normality condition, there exist a unitary matrix $U$ such that  $U^*AU$ and $U^*X^*U$ are both diagonal. A direct computation shows that: $$\begin{pmatrix}U^*&0\\0&U ^* \end{pmatrix}\begin{pmatrix}A&X\\X^*&B    \end{pmatrix}\begin{pmatrix}U&0\\0&U    \end{pmatrix}=\begin{pmatrix}U^*AU&U^*XU\\U^*X^*U&U^*BU    \end{pmatrix}=\mathcal{G}.$$ Now  $U^*XU$ also commute with $U^*BU,$  since $U^*XU$ is diagonal and all of its diagonal entries are distinct by \rref{ff} $U^*BU$ must be also diagonal, applying \tref{pww} to the matrix $\mathcal{G}$ yields to:$$\left\lVert \begin{pmatrix}A&X\\X^*&B    \end{pmatrix}\right\rVert=\|\mathcal{G
}\|\le \| U^*AU+U^*BU            \|=\|A+B    \|,$$
for all symmetric norms. The inequality holds for any  $X$ normal  by a continuity argument.
\end{proof}
\begin{lemma}\label{plp}
Let $$N=\begin{pmatrix}\begin{psmallmatrix}a_{1} & 0 &\cdots & 0\\ 0 & a_{2} & \cdots & 0\\ \vdots & \vdots & \ddots &\vdots \\
0& 0 & \cdots & a_{n} \end{psmallmatrix}&D\\D^*&\begin{psmallmatrix}b_{1} & 0 &\cdots & 0\\ 0 & b_{2} & \cdots & 0\\ \vdots & \vdots & \ddots &\vdots \\
0& 0 & \cdots & b_{n} \end{psmallmatrix}\end{pmatrix}$$ where $a_1,\cdots, a_{n}$ respectively $b_{1},\cdots, b_{n}$ are  nonnegative respectively  negative real numbers, $A=\begin{psmallmatrix}a_{1} & 0 &\cdots & 0\\ 0 & a_{2} & \cdots & 0\\ \vdots & \vdots & \ddots &\vdots \\0& 0 & \cdots & a_{n} \end{psmallmatrix},$ $B=\begin{psmallmatrix}b_{1} & 0 &\cdots & 0\\ 0 & b_{2} & \cdots & 0\\ \vdots & \vdots & \ddots &\vdots \\
0& 0 & \cdots & b_{n} \end{psmallmatrix}$ and $D$ is any diagonal matrix, then nor $N$  neither $-N$ is  positive semi-definite. Set  $(d_1,\cdots,d_n)$ as the diagonal entries of $D^*D, $ if $a_{i}+b_{i}\ge 0$ and  $a_ib_i-d_i < 0$ for all $i\le n,$              then  $\|N\|>\|A+B\|.$ for all symmetric norms
\end{lemma}
\begin{proof}
The diagonal  of $N$ has  negative and   positive  numbers, thus nor $N$  neither $-N$ is  positive semi-definite, now any two diagonal matrices will commute, in particular $D^*$ and $A,$ by applying \tref{sinj} we get that  the eigenvalues of $N$ are the roots of $$\det((A-{\mu}I_{n})(B-{\mu}I_{n})-D^*D)=0$$
Equivalently the eigenvalues are all the  solutions  of the $n$ equations:$$ \begin{array}{rlrcl}1)&&({a}_1-{\mu})({b}_1-{\mu})-d_{1}&=&0
\\[0.1cm]2)& &({a}_2-{\mu})({b}_2-{\mu})-d_{2}&=&0
\\[0.1cm]3)&&\text{  }({a}_3-{\mu})({b}_3-{\mu})-d_{3}&=&0
\\\vdots{\text{ }}&&&\vdots&\\
i)&&({a}_i-{\mu})({b}_i-{\mu})-d_{n}&=&0\\\vdots\text{  }&&&\vdots&\\[0.1cm]
n)&&({a}_n-{\mu})({b}_n-{\mu})-d_{n}&=&0
\end{array}\indent (S)\\[0.1cm]$$
 Let us denote by $x_i$ and $y_i$  the two solutions of the {$i^{th}$} equation then:
$$\begin{array}{rclcl}x_1+y_1&=&a_1+b_1&\ge&0\\[0.1cm]
x_2+y_2&=&a_2+b_2&\ge& 0\\ &\vdots&&\vdots&\\[0.1cm]
x_n+y_n&=&a_n+b_n&\ge &0\\[0.1cm]
\end{array}\indent \begin{array}{rclcl}x_1y_1&=&a_1b_1-d_1&<&0\\[0.1cm]
x_2y_2&=&a_2b_2-d_2&<& 0\\ &\vdots& &\vdots&\\[0.1cm]
x_ny_n&=&a_nb_n-d_n&< &0\\[0.1cm]
\end{array} $$
This implies that each equation of $(S)$ has one negative  and one positive solution,  their sum is positive,  thus the positive root is bigger or equal than the negative one. Since  $A+B=\begin{psmallmatrix}a_{1}+b_{1} &\cdots & 0\\ \vdots &  \ddots &\vdots \\
0&  \cdots &a_{n}+ b_{n} \end{psmallmatrix},$  summing over  indexes we see that $\|N\|_k>\|A+B\|_k$ for $k=1,\cdots,n$ which yields to   $\|N\|>\|A+B\|$ for all symmetric norms
\end{proof}
It seems easy to construct examples of  non P.S.D matrices $N$ written in blocks such that   $\displaystyle
 {\|N\|}_{s}>{\|A+B\|}_{s},$ let us have a  look of such inequality for P.S.D. matrices.
\begin{example} Let $$C=\begin{pmatrix} \dfrac{4}{3} & 0&1&-1\\0&1&0&\dfrac{1}{5}\\
1& 0& \dfrac{3}{2}&0\\-1&\dfrac{1}{5}& 0&{2}\end{pmatrix}=\begin{pmatrix}A&X\\X^*&B\end{pmatrix}$$  where $A=\begin{pmatrix}\frac{4}{3} &0\\ 0&1\end{pmatrix},$   $B=\begin{pmatrix}\frac{3}{2}&0\\0&2\end{pmatrix}.$ Since the eigenvalues of $C$ are all positive with  $\displaystyle{\lambda}_1\approx 3.008,\text{   } {\lambda}_2\approx 1.7 ,\text{       }{\lambda}_3\approx 0.9, \text{   }{\lambda}_4\approx 0.089$, $C$ is positive definite  and we verify  that $$3.008\approx {\|C\|}_s>{\|A+B\|}_s=3$$
\end{example}
\begin{example}
Let $$N_y=\begin{pmatrix}  2& 0&0&2\\0&y&0&0\\0& 0& 1&0\\2&0& 0&{2}\end{pmatrix}=\begin{pmatrix}A&X\\X^*&B\end{pmatrix}$$ where  $A=\begin{pmatrix}2 &0\\ 0&y\end{pmatrix}$ and    $B=\begin{pmatrix}1&0\\0&2\end{pmatrix}.$   The eigenvalues of $N_y$ are the numbers: $\displaystyle{\lambda}_1=4,\text{   } {\lambda}_2=1,\text{       }{\lambda}_3=y, \text{   }{\lambda}_4=0,$ thus if $y\ge 0,$ $N_y$ is positive semi-definite  and for all $ y$ such that   $0\le y < 1$ we have
\begin{enumerate}
\item $4={\|N_y\|}_s>{\|A+B\|}_s=3$
\item $ 16 +y^2+1={\| N\|}_{(2)}^2>{\|A+B\|}_{(2)}^2=4(3+y)+y^2+1$
\end{enumerate}
\end{example}

\end{document}